\documentclass[11pt]{amsart}
\usepackage{perpage,amsfonts,amssymb,amsmath,amsfonts,amsthm,enumerate}

\theoremstyle{plain}
\newtheorem{theorem}{Theorem}
\newtheorem*{theoremnn}{Theorem}

\newtheorem{corollary}[theorem]{Corollary}

\newtheorem{lemma}[theorem]{Lemma}
\newtheorem{proposition}[theorem]{Proposition}

\theoremstyle{remark}
\newtheorem*{remark}{\textbf{Remark}}

\numberwithin{equation}{section}

\MakePerPage{footnote} \allowdisplaybreaks 

\newcommand\R{\mathbb R}

\title[Sharp HLS inequalities on the Heisenberg group]
{Existence of maximizers for Hardy-Littlewood-Sobolev inequalities on the Heisenberg group}

\author{Xiaolong Han}
\address{Department of Mathematics, Wayne State University, Detroit, MI 48202, USA}
\email{xlhan@math.wayne.edu}



\begin{document}
\maketitle

\begin{abstract}
In this paper, we investigate the sharp Hardy-Littlewood-Sobolev inequalities on the Heisenberg group. On one hand, we apply the concentration compactness principle to prove the existence of the maximizers. While the approach here gives a different proof under the special cases discussed in a recent work of Frank and Lieb \cite{FL2}, we generalize the result to all admissible cases. On the other hand, we provide the upper bounds of sharp constants for these inequalities.
\end{abstract}

\section{Introduction}

The $n$-dimensional Heisenberg group is $\mathbb H^n=\mathbb
C^n\times\R$ with group structure given by
$$uv=(z,t)(z',t')=(z+z',t+t'+2Im(z\cdot \overline z))$$
for any two points $u=(z,t),v=(z',t')\in\mathbb H^n$, where
$z,z'\in\mathbb C^n$, $t,t'\in\R$, and $z\cdot\overline
z=\sum_{j=1}^n z_j\overline{z'_j}$. Haar measure on $\mathbb H^n$ is the Lebesgue measure $du=dzdt$, in which $z=x+iy$ with $x,y\in\R^n$.

The Lie algebra of $\mathbb H^n$ is generated by the left invariant
vector fields
$$T=\frac{\partial}{\partial t},
\ X_j=\frac{\partial}{\partial x_j}+2y_j\frac{\partial}{\partial
t},\ Y_j=\frac{\partial}{\partial y_j}-2x_j\frac{\partial}{\partial
t}.$$

For each real number $d\in\R$, we denote the dilation
$\delta_du=\delta_d(z,t)=(dz,d^2t)$, the homogeneous norm on
$\mathbb H^n$ as $|u|=|(z,t)|=(|z|^4+t^2)^{1/4}$, and $Q=2n+2$ as
the homogeneous dimension.

We recall the famous Hardy-Littlewood-Sobolev (shorted as HLS in the following context) inequality on $\R^N$: Let $1<r,s<\infty$ and
$0<\lambda<N$ such that $\frac1r+\frac1s+\frac\lambda N=2$, then
\begin{equation}
\left|\int\int_{\R^N\times\mathbb
R^N}\frac{\overline{f(x)}g(y)}{|x-y|^\lambda}dxdy\right|\leq
C\|f\|_r\|g\|_s
\end{equation}
for any $f\in L^r(\R^N)$ and $g\in L^s(\R^N)$, where $\|\cdot\|_r$
and $\|\cdot\|_s$ are the $L^r$ and $L^s$ norms on $\R^N$,
respectively. And $0<C<\infty$ is a constant depending on $r$, $\lambda$, and $N$ only.

This inequality was introduced by Hardy and Littlewood \cite{HL1,
HL2, HL3} on $\R^1$ and generalized by Sobolev \cite{S} to $\R^N$. We denote by $C_{r,\lambda,N}$ the sharp (best) constant that we can put into (1.1), finding and proving sharp constant $C_{r,\lambda,N}$ and its maximizers\footnote{They are also referred as optimizers or extremals in some literature.} (functions which, when inserted into (1.1), the equality holds with the smallest constant $C_{r,\lambda,N}$) have driven a lot people's attention. In Lieb's paper \cite{Li}, existence of the maximizers was proved. Furthermore, when $r=s=2N/(2N-\lambda)$, he gave explicit formulae of sharp constants $C_{\lambda,n}$ and maximizers. Precisely,
\begin{theoremnn}[Lieb \cite{Li}]
Let $1<r,s<\infty$, $0<\lambda<N$, and $\frac1r+\frac1s+\frac\lambda
N=2$, then there exists a sharp constant $C_{p,\lambda,N}$,
maximizers of $f\in L^r(\R^N)$ and $g\in L^s(\R^N)$ such that
\begin{equation}
\left|\int\int_{\R^N\times\mathbb
R^N}\frac{\overline{f(x)}g(y)}{|x-y|^\lambda}dxdy\right|=
C_{r,\lambda,N}\|f\|_r\|g\|_s.
\end{equation}

If $r=s=2N/(2N-\lambda)$, then
\begin{equation*}
C_{r,\lambda,N}=C_{\lambda,N}=\pi^{\lambda/N}\frac{\Gamma(N/2-\lambda/2)}{\Gamma(N-\lambda/2)}
\left(\frac{\Gamma(N/2)}{\Gamma(N)}\right)^\frac{\lambda-N}{N}.
\end{equation*}

In this case (1.2) holds if and only if $f\equiv
(\text{const.})g$ and
\begin{equation}
f(x)=\frac{c}{\bigg[d+|x-a|^2\bigg]^{(2N-\lambda)/2}}
\end{equation}
for some $c\in\mathbb C$, $0\ne d\in\R$, and $a\in\R^N$.
\end{theoremnn}

\begin{remark}\hfill
\begin{enumerate}[1.]
\item The proof of the above theorem can also be found in Lieb and
Loss's monograph \cite{LL} with more details, in which they also
proved that the sharp constant $C_{r,\lambda,N}$ satisfies
\begin{equation}
C_{r,\lambda,N}\le\frac{N}{rs(N-\lambda)}\left(\frac{\omega_{N-1}}{N}\right)^{\lambda/N}\left[\left(\frac{\lambda/N}{1-1/r}\right)^\frac\lambda
N+\left(\frac{\lambda/N}{1-1/s}\right)^\frac\lambda N\right],
\end{equation}
where $\omega_{N-1}$ is the area of unit sphere in $\R^N$, i.e.,
$\omega_{N-1}=2\pi^{N/2}/\Gamma(N/2)$. The original proof by Lieb \cite{Li} applies rearrangement methods, a new rearrangement-free proof was provided by Frank and Lieb \cite{FL1, FL3}.

\item The existence of maximizers were also proved by Lions (\S 2.1 in \cite{Lio4}), which is an application of the concentration compactness principle introduced by him in a series of papers \cite{Lio1, Lio2, Lio3, Lio4}. See also \S II.4 in \cite{Str} about the application on sharp Sobolev inequalities.

\item The uniqueness of maximizers (1.3) was proposed by Lieb \cite{Li} as an open problem, and was answered by Chen, Li,
and Ou \cite{CLO2}, in which they used moving plane method for
integral equations, (A different approach using moving sphere method for integral equations has been done by Li \cite{L}.) a related work by Chen, Li, and Ou \cite{CLO1} studied the integral systems using the similar method. The formula of the maximizers (after dilations and translations) assume
$$\frac{1}{\bigg[1+|x|^2\bigg]^{(2N-\lambda)/2}}.$$

\item We shall point out that (1.4) is not sharp, and neither sharp
constant $C_{r,\lambda,N}$ nor maximizers are known yet when $r\ne s$.
\end{enumerate}
\end{remark}

The analogous HLS inequality on the Heisenberg group was announced by Stein \cite{St} and proved by Folland and Stein \cite{FS} in terms of fractional integral (Proposition 8.7 and Lemma 15.3 in \cite{FS}).

\begin{theoremnn}[Folland \& Stein \cite{FS}]
Let $1<r,s<\infty$, $0<\lambda<Q$, and $\frac1r+\frac1s+\frac\lambda Q=2$, then there exists a constant $C$, independent of
$f\in L^r(\mathbb H^n)$ and $g\in L^s(\mathbb H^n)$, such that
\begin{equation}
\left|\int\int_{\mathbb H^n\times\mathbb
H^n}\frac{\overline{f(u)}g(v)}{|u^{-1}v|^\lambda}dudv\right|\leq
C\|f\|_r\|g\|_s.
\end{equation}
\end{theoremnn}

Here $u=(z,t)$ and $v=(z',t')$, $u^{-1}=(-z,-t)$, and
$d(u,v):=|u^{-1}v|=|uv^{-1}|$ is a left-invariant quasi-metric (cf.
Section 4 in \cite{N}). And without causing any confusion, we denote
$\|\cdot\|_r$ and $\|\cdot\|_s$ as the $L^r$ and $L^s$ norms on
$\mathbb H^n$.

Towards the sharp version of (1.5), Jerison and Lee \cite{JL}
provided sharp constant and maximizer when $\lambda=Q-2$ and
$p=q=2Q/(2Q-\lambda)=2Q/(Q+2)$. Very recently, Frank and Lieb
\cite{FL2} generalized their results to all $0<\lambda<Q$ as the
following theorem.

\begin{theoremnn}[Frank \& Lieb \cite{FL2}]
Let $0<\lambda<Q$ and $r=2Q/(2Q-\lambda)$, then for any $f,g\in
L^r(\mathbb H^n)$
\begin{equation*}
\left|\int\int_{\mathbb H^n\times\mathbb
H^n}\frac{\overline{f(u)}g(v)}{|u^{-1}v|^\lambda}dudv\right|\leq
\left(\frac{\pi^{n+1}}{2^{n-1}n!}\right)^{\frac\lambda
Q}\frac{n!\Gamma((Q-\lambda)/2)}{\Gamma^2((2Q-\lambda)/4)}\|f\|_r\|g\|_r,
\end{equation*}
with equality if and only if
$$f(u)=cH(d(a^{-1}u)),\ g(v)=c'H(d(a^{-1}v))$$
for some $c,c'\in\mathbb C$, $d>0$, $a\in\mathbb H^n$ (unless
$f\equiv 0$ or $g\equiv0$), and
$$H=\frac{1}{\bigg[(1+|z|^2)^2+t^2\bigg]^{(2Q-\lambda)/4}}.$$
\end{theoremnn}

Their results also justified Branson, Fontana, and Morpurgo's guess
\cite{BFM} about the maximizer $H$. However, little about maximizers and sharp constants are known when $r\ne s$. For more information, we refer to Frank and Lieb \cite{FL2} with some historical note on this subject. In \cite{HLZ}, two weighted versions of (1.5) are studied, whose maximizers are also investigated therein. Some other related results concerning the sharp constants of Moser-Trudinger inequalities on the Heisenberg group and $\mathbb C^n$ are \cite{CL1, CL2}.

The major theorem of this paper is to prove the existence of
maximizers for (1.5) when $r\neq s$, and we state in terms of
fraction integral form: Define for $f\in L^p(\mathbb H^n)$
$$I_\lambda(f)(u)=\int_{\mathbb H^n}\frac{f(v)}{|u^{-1}v|^\lambda}dv,$$
for $u\in\mathbb H^n$. Then, we have the sharp constant
\begin{equation}
C_{p,\lambda,n}=\sup_{\|f\|_p=1}\|I_\lambda(f)\|_q<\infty
\end{equation}
under the condition that
\begin{equation}
\frac1q=\frac1p-\frac{Q-\lambda}{Q}.
\end{equation}

We have the following maximization theorem, and one only needs to apply H\"{o}lder's inequality to solve the maximizing problem for
(1.5) (by plugging in $s=p$ and $r=q'$).

\begin{theorem}[Existence of maximizers]
Let $\{f_j\}$ be a maximizing sequence of problem (1.6) and (1.7),
then there exists $\{u_j\}\subseteq\mathbb H^n$ and
$\{d_j\}\subseteq(0,\infty)$ such that the new maximizing
sequence $\{h_j\}$ defined by
\begin{equation*}
h_j(u)=\frac{1}{d_j^{Q/p}}f_j\left(\frac{u_ju}{d_j}\right)
\end{equation*}
is relatively compact in $L^p(\mathbb H^n)$. In particular, there
exists a maximum of (1.6) and (1.7).
\end{theorem}

\begin{remark}
Under a special case when $p=2Q/(2Q-\lambda)$ and $q=2Q/\lambda$, we derive from the above theorem a different approach to prove the existence of maximizers as shown in \S 4 of \cite{FL2}.
\end{remark}

Having confirmed existence of maximizers, furthermore we give upper bounds for sharp constants as follows.

\begin{theorem}[Upper bounds of sharp constants]
In (1.5), $C$ can be chosen as
\begin{equation*}
\frac{Q|B_1(0)|^\frac\lambda
Q}{rs(Q-\lambda)}\left[\left(\frac{\lambda/Q}{1-1/r}\right)^\frac\lambda
Q+\left(\frac{\lambda/Q}{1-1/s}\right)^\frac\lambda Q\right],
\end{equation*}
in which $B_1(0)$ is the unit Heisenberg ball, that is,
$B_1(0)\subseteq\mathbb H^n=\{u\in\mathbb H^n||u|<1\}$ with $|B_1(0)|$ as
its volume. Precisely, from \cite{CL1},
$$|B_1(0)|=\frac{2\pi^\frac{Q-2}{2}\Gamma(1/2)\Gamma((Q+2)/4)}{(Q-2)\Gamma((Q-2)/2)\Gamma((Q+4)/4)}.$$
\end{theorem}

The following two sections are devoted to proving Theorems 1 and 2, respectively.

\section{Proof of Theorem 1}

Let us first outline the difficulties to be faced, the loss of
compactness is caused by a large group of actions consisting of
dilations and translations. One can use symmetrization to exclude
some actions and ensure the existence of maximizers on $\mathbb R^N$. However, symmetrization can not be expected to work on $\mathbb H^n$ because of its dilation structure. Therefore, different approach is needed to study the compactness here. It is worthwhile to remark now that in the process we frequently extract to subsequences of the maximizing sequence as needed. 

A crucial lemma that refines Fatou's lemma is due to Br\'ezis and Lieb \cite{BL}.

\begin{lemma}[Br\'ezis-Lieb lemma]
Let $0<p<\infty$, $\{f_j\}\subseteq L^p(\mathbb H^n)$ satisfying
$\|f\|_p\le C$ and $f_j\rightarrow f$ a.e., then
$$\lim_{j\rightarrow\infty}\int_{\mathbb H^n}\bigg||f_j(u)|^p-|f(u)-f_j(u)|^p-|f(u)|^p\bigg|du=0.$$
\end{lemma}

Now suppose that we are given a maximizing sequence $\{f_j\}$ for (1.6) and (1.7), and without loss of generality we assume that $\|f_j\|_p=1$, our goal is to generate $\{h_j\}$ as stated in Theorem 1. To recover from the loss of compactness from dilations and translations, we first need the following concentration compactness lemma from \cite{Lio1}, we provide a proof here for completeness.

\begin{lemma}
For simplicity, we denote by $\rho_j=|f_j|^p$ as a nonnegative measure on $\mathbb H^n$, thus $\int_{\mathbb H^n}\rho_j=1$. Then, there exists a subsequence of $\{\rho_j\}$ (and we still denote as $\{\rho_j\}$) such that one of the following holds.

\begin{itemize}
\item[(i)] For all $R>0$, we have
$$\lim_{j\rightarrow\infty}\left(\sup_{u\in\mathbb H^n}\int_{B_R(u)}\rho_j\right)=0.$$
\item[(ii)] There exists $\{u_j\}\subseteq\mathbb H^n$ such that for each $\epsilon>0$ small enough, we can find $R_0>0$ with
$$\int_{B_{R_0}(u_j)}\rho_j\ge1-\epsilon$$
for all $j\in\mathbb N$.
\item[(iii)] There exists $0<k<1$ such that for each $\epsilon>0$ small enough, we can find $R_0>0$ and $\{u_j\}\subseteq\mathbb H^n$ such that given any $R\ge R_0$, there exist $\rho_j^1$ and $\rho_j^2$ as two nonnegative measures satisfying 
\begin{enumerate}[1.~]
\item $\rho^1_j+\rho^2_j=\rho_j$.
\item $\text{supp}(\rho^1_j)\subseteq B_R(u_j)$ and $\text{supp}(\rho^2_j)\subseteq B^c_R(u_j)$.
\item $$\limsup_{j\rightarrow\infty}\left(\left|k-\int_{\mathbb H^n}\rho_j^1\right|+\left|(1-k)-\int_{\mathbb H^n}\rho_j^2\right|\right)\le\epsilon.$$
\end{enumerate}
\end{itemize}
\end{lemma}

\begin{proof}[Proof of Lemma 4]

We define the Levy concentration function for $\rho_j$ on $\mathbb H^n$ as
$$Q_j(R)=\sup_{u\in\mathbb H^n}\int_{B_R(u)}\rho_j$$
for $R\in[0,\infty]$. It is obvious that $Q_j\in\text{BV}[0,\infty]$ is nonnegative and nondecreasing with 
$$Q_j(0)=0\text{, and }Q_j(\infty)=1$$ 
for all $j\in\mathbb N$. Therefore, we can find a nonnegative and nondecreasing function $Q\in\text{BV}[0,\infty]$ such that by passing to a subsequence of $\{Q_j\}$ if necessary (and we still denote without causing any confusion by $\{Q_j\}$)
$$\lim_{j\rightarrow\infty}Q_j(R)=Q(R)$$
for all $R\in[0,\infty)$.

Now we write
$$k=\lim_{R\rightarrow\infty}Q(R),$$
and thus $0\le k\le1$.

\begin{itemize}
\item[(i)] If $k=0$, then easily we have
$$\lim_{j\rightarrow\infty}\left(\sup_{u\in\mathbb H^n}\int_{B_R(u)}\rho_j\right)=0$$
for all $R>0$.

\item[(ii)] If $k=1$, then we first choose $R_1>0$ such that $Q(R_1)>\frac34$, and for fixed $0<\epsilon<\frac14$, we choose $R_2$ such that $Q(R_2)>1-\frac\epsilon2>\frac34$. Because 
$$Q_j(R)=\sup_{u\in\mathbb H^n}\int_{B_R(u)}\rho_j,$$
we let $u_j,v_j\in\mathbb H^n$ satisfy
$$\int_{B_{R_1}(u_j)}\rho_j\ge Q_j(R_1)-\frac1j$$
and 
$$\int_{B_{R_2}(v_j)}\rho_j\ge Q_j(R_2)-\frac1j$$
for all $j\in\mathbb N$. We compute that
\begin{eqnarray*}
&&\int_{B_{R_1}(u_j)}\rho_j+\int_{B_{R_2}(v_j)}\rho_j\\
&\ge&Q_j(R_1)+Q_j(R_2)-\frac2j+o(1)\\
&>&1\\
&=&\int_{\mathbb H^n}\rho_j
\end{eqnarray*}
for $j$ large, which means that $B_{R_1}(u_j)\cap B_{R_2}(v_j)\ne\emptyset$. Therefore, 
$$B_{R_2}(v_j)\subseteq B_{R_1+2R_2}(u_j),$$
in which we use the fact that the quasi-metric defined as $d(u,v)=|u^{-1}v|$ on $\mathbb H^n$ is a metric, see \cite{C}, also \S4 in \cite{SW}, and we continue to take advantage of this in the following when estimating distances. Now compute that
$$\int_{B_{R_1+2R_2}(u_j)}\rho_j\ge Q_j(R_2)-\frac1j\ge Q(R_2)+o(1)-\frac1j\ge1-\epsilon$$
for $j>j(\epsilon)$. Furthermore, we select $R_3$ such that
$$\int_{B_{R_3}(0)}\rho_j\ge1-\epsilon$$
for $j=1,2,...,j(\epsilon)$. Then, one arrives at the conclusion in (ii) by taking $R_0=R_1+2R_2+R_3$. 

\item[(iii)] If $0<k<1$, then $\forall\epsilon>0$, choose $R_0$ such that $Q(R_0)>k-\frac\epsilon8$. For $j>j(\epsilon)$, we have
$$k-\frac\epsilon4<Q_j(R_0)<k+\frac\epsilon4,$$
and therefore, there is $\{u_j\}$ such that
$$\int_{B_{R_0}(u_j)}\rho_j>k-\frac\epsilon2.$$

Similarly, we can enlarge $j(\epsilon)$ if necessary to get a sequence $\{R_j\}$ with $R_j\rightarrow\infty$ such that
$$\int_{B_{R_j}(u_j)}\rho_j<k+\frac\epsilon2$$
for all $j>j(\epsilon)$.

For any given $R\ge R_0$, we may assume $R\le R_j$ for all $j\in\mathbb N$. This means that there exists $\{u_j\}\subseteq\mathbb H^n$ such that
$$k-\frac\epsilon2\le\int_{B_{R_0}(u_j)}\rho_j\le\int_{B_{R}(u_j)}\rho_j\le\int_{B_{R_j}(u_j)}\rho_j\le k+\frac\epsilon2.$$

Set
$$\rho_j^1=\rho_j\chi_{B_{R}(u_j)}\text{ and }\rho_j^2=\rho_j\chi_{B^c_{R}(u_j)},$$
thus,
\begin{eqnarray*}
&&\left|k-\int_{\mathbb H^n}\rho_j^1\right|+\left|(1-k)-\int_{\mathbb H^n}\rho_j^2\right|\\
&=&\left|k-\int_{B_{R}(u_j)}\rho_j\right|+\left|(1-k)-\int_{B^c_{R}(u_j)}\rho_j\right|\\
&\le&\epsilon.
\end{eqnarray*}
\end{itemize}
\end{proof}

\begin{remark}
Back to our maximizing problem, let $\{f_j\}\subseteq L^p(\mathbb H^n)$ be a maximizing sequence of (1.6) and (1.7) satisfying $\|f_j\|=1$, then with the help of dilations (as $\{d_j\}$ in Theorem 1), we can always assume that 
\begin{equation*}
Q_j(1)=\sup_{u\in\mathbb H^n}\int_{B_1(u)}\rho_j=\frac12
\end{equation*}
as defined in the proof of Lemma 4 without affecting our maximizing problem since (1.6) is dilation-invariant, therefore we are able to eliminate the case in (i). Next we prove that the case in (iii) can not happen, either, from which we only need to focus on the case in (ii).
\end{remark}

\begin{proposition}
Let $\{f_j\}\subseteq L^p(\mathbb H^n)$ be a maximizing sequence of (1.6) and (1.7) satisfying $\|f_j\|=1$, then (iii) in Lemma 4 can not occur.
\end{proposition}

\begin{proof}[Proof of Proposition 5]
We argue by contradiction. If (iii) in Lemma 4 occurs, then there exist $0<k<1$ and a subsequence of $\{f_j\}$ (which we still denote by $\{f_j\}$) such that for each $\epsilon>0$ small enough, we can find $R_0>0$ and $\{u_j\}\subseteq\mathbb H^n$ such that given any $R\ge R_0$,
$$\|f_j\chi_{B_R(0)}\|^p_p=k+O(\epsilon)\text{ and }\|f_j\chi_{B_R^c(0)}\|^p_p=1-k+O(\epsilon).$$

Without loss of generality, we may assume $u_j=0$ for all $j\in\mathbb Z$ since (1.6) is translation-invariant. Thus, for any $u\in\mathbb H^n$, let $R=j|u|$ for $j\ge j(\epsilon,|u|)$ such that $j|u|>R_0$, we observe that $|u|\le\frac1j|v|$ for all $v\in B_R^c(0)$, then
$$|u^{-1}v|\ge|v|-|u|\ge\frac{j-1}{j}|v|,$$
and therefore,
\begin{eqnarray*}
&&|I_\lambda(f_j)(u)-I_\lambda(f_j\chi_{B_R(0)})(u)|\\
&=&|I_\lambda(f_j\chi_{B_R^c(0)})(u)|\\
&=&\left|\int_{\mathbb H^n}\frac{(f_j\chi_{B_R^c(0)})(v)}{|u^{-1}v|^\lambda}dv\right|\\
&\le&\left(\int_{B_R^c(0)}|f_j(v)|^pdv\right)^{1/p}\left(\int_{B_R^c(0)}|u^{-1}v|^{-\lambda p'}dv\right)^{1/p'}\\
&\le&C\left(\frac{j}{j-1}\right)^\lambda\left(\int_{j|u|}^\infty r^{Q-\lambda p'-1}dr\right)^{1/p'}\\
&\le&C\left(\frac{j}{j-1}\right)^\lambda\left(\frac{1}{\lambda p'-Q}\right)^{1/p'}\bigg(j|u|\bigg)^{(Q-\lambda p')/p'}\\
&\rightarrow&0
\end{eqnarray*}
as $j\rightarrow\infty$. Here, $C$ depends only on $\mathbb H^n$, and the integral is finite because $Q<\lambda p'$ from (1.7):
$$\frac1q=\frac1p-\frac{Q-\lambda}{Q}>0.$$

Now we apply the Br\'ezis-Lieb lemma in Lemma 3 because $I_\lambda(f_j)\rightarrow I_\lambda(f_j\chi_{B_R(0)})$ a.e. and get
$$\|I_\lambda(f_j)\|^q_q=\|I_\lambda(f_j\chi_{B_R(0)})\|^q_q+\|I_\lambda(f_j\chi_{B_R^c(0)})\|^q_q+o(1),$$
in which the left-hand side goes to $C_{p,\lambda,n}^q$ since $\{f_j\}$ maximizes (1.6), while the right-hand side
\begin{eqnarray*}
&&\|I_\lambda(f_j\chi_{B_R(0)})\|^q_q+\|I_\lambda(f_j\chi_{B_R^c(0)})(u)\|^q_q+o(1)\\
&\le&C_{p,\lambda,n}^q\|f_j\chi_{B_R(0)}\|^q_p+C_{p,\lambda,n}^q\|f_j\chi_{B_R^c(0)}\|^q_p+o(1)\\
&\le&C_{p,\lambda,n}^q(k+O(\epsilon))^{\frac qp}+C_{p,\lambda,n}^q(1-k+O(\epsilon))^{\frac qp}+o(1)\\
&\le&C_{p,\lambda,n}^q\left[k^{\frac qp}+(1-k)^{\frac qp}\right]+O(\epsilon)+o(1)\\
&<&C_{p,\lambda,n}^q,
\end{eqnarray*}
if $0<k<1$ for large $j$ because $\frac qp>1$, and we conclude the contradiction.
\end{proof}

We can now proceed under the scope of (ii) in Lemma 4:
There exists $\{u_j\}\subseteq\mathbb H^n$ such that for $R$ large, we have
$$\int_{B_R(u_j)}|f_j|^p\ge1-\epsilon(R).$$

Due to translations $f_j(v)\rightarrow f_j(u_jv)$, we use $\{f_j\}$ to denote the new maximizing sequence satisfying
\begin{equation}
\int_{B_R(0)}|f_j|^p\ge1-\epsilon(R),
\end{equation}
and we derive the following corollary, in which the byproduct (2.2) serves as an important ingredient in the proof of Theorem 1.
 
\begin{corollary}
Let $\{f_j\}\subseteq L^p(\mathbb H^n)$ be a maximizing sequence of (1.6) and (1.7) satisfying $\|f_j\|=1$ and
$$\int_{B_R(0)}|f_j|^p\ge1-\epsilon(R),$$ 
we may assume that $f_j\rightarrow f$ weakly in $L^p(\mathbb H^n)$ (by passing to a subsequence if necessary). Then, by passing to a subsequence again if necessary,
$$I_\lambda(f_j)\rightarrow I_\lambda(f)\text{ a.e.}.$$
\end{corollary}

\begin{proof}[Proof of Corollary 6]
We show that $I_\lambda(f_j)\rightarrow I_\lambda(f)$ in measure to ensure the existence of a pointwisely convergent subsequence of $\{f_j\}$. Observe that for $M>R$,
\begin{eqnarray*}
&&\|I_\lambda(f_j)\chi_{B_M^c(0)}\|_q\\
&\le&\|I_\lambda(f_j\chi_{B_R(0)})\chi_{B_M^c(0)}\|_q+\|I_\lambda(f_j\chi_{B_R^c(0)})\chi_{B_M^c(0)}\|_q\\
&\le&\|I_\lambda(f_j\chi_{B_R(0)})\chi_{B_M^c(0)}\|_q+C_{p,\lambda,n}\|f_j\chi_{B_R^c(0)}\|_p\\
&\le&\|I_\lambda(f_j\chi_{B_R(0)})\chi_{B_M^c(0)}\|_q+\epsilon(R),\end{eqnarray*}
in which we apply Minkowski's integral inequality to estimate the first term, noticing that $|u^{-1}v|\ge|u|-R$ for $|v|\le R<M\le|u|$,\begin{eqnarray*}
&&\|I_\lambda(f_j\chi_{B_R(0)})\chi_{B_M^c(0)}\|_q\\
&=&\left(\int_{B_M^c(0)}|I_\lambda(f_j\chi_{B_R(0)})(u)|^qdu\right)^{1/q}\\
&=&\left(\int_{|u|\ge M}\left|\int_{|v|\le R}\frac{f_j(v)}{|u^{-1}v|^\lambda}dv\right|^qdu\right)^{1/q}\\
&\le&\|f_j\chi_{B_R(0)}\|_1\left(\int_{|u|\ge M}\frac{1}{(|u|-R)^{\lambda q}}du\right)^{1/q}\\
&\le&C(R,p,n)(M-R)^{(Q-\lambda q)/q}\\
&\rightarrow&0
\end{eqnarray*}
for every fixed $R$ as $M\rightarrow\infty$ since $Q<\lambda q$ from (1.7). Therefore, we have
\begin{equation}
\|I_\lambda(f_j)\chi_{B_M^c(0)}\|_q\le\epsilon(M),
\end{equation}
and that is,
$$\|I_\lambda(f_j)-I_\lambda(f_j)\chi_{B_M(0)}\|_q\le\epsilon(M).$$

Since $f_j\rightarrow f$ weakly in $L^p(\mathbb H^n)$, we have 
$$\|f\chi_{B_R^c(0)}\|_p^p\le\liminf_{j\rightarrow\infty}\|f_j\chi_{B_R^c(0)}\|_p^p\le\epsilon(R).$$

Similarly, one can derive for $f$,
$$\|I_\lambda(f)-I_\lambda(f)\chi_{B_M(0)}\|_q\le\epsilon(M).$$

Therefore, given $k>0$,
\begin{eqnarray}
&&|\{|I_\lambda(f_j)(u)-I_\lambda(f)(u)|\ge15k|\}|\nonumber\\
&\le&|\{|I_\lambda(f_j)(u)-I_\lambda(f_j)(u)\chi_{B_M(0)}(u)|\ge5k\}|+\nonumber\\
&&|\{|I_\lambda(f_j)(u)\chi_{B_M(0)}(u)-I_\lambda(f)(u)\chi_{B_M(0)}(u)|\ge5k\}|+\nonumber\\
&&|\{|I_\lambda(f)(u)\chi_{B_M(0)}(u)-I_\lambda(f)(u)|\ge5k\}|\nonumber\\
&\le&2\left[\frac{\epsilon(M)}{5k}\right]^q+|\{|I_\lambda(f_j)(u)-I_\lambda(f)(u)|\ge5k\}\cap B_M(0)|.
\end{eqnarray}

Thus, it remains to estimate the second term above. Denote 
$$I_\lambda^\eta(f)(u)=\int_{B_\eta^c(u)}\frac{f(v)}{|u^{-1}v|^\lambda}dv,$$
then
$$I_\lambda^\eta(f_j\chi_{B_R(0)})(u)\rightarrow I_\lambda^\eta(f\chi_{B_R(0)})(u)$$
for all $u\in\mathbb H^n$ because
$$|u^{-1}v|^{-\lambda}\chi_{B_R(0)}\chi_{B_\eta^c(u)}\in L^{p'}(\mathbb H^n)$$
for any fixed $u\in\mathbb H^n$ and $\eta>0$. Therefore, $I_\lambda^\eta(f_j\chi_{B_R(0)})\rightarrow I_\lambda^\eta(f\chi_{B_R(0)})$ locally in measure, which means, 
\begin{equation}
\left|\{|I_\lambda^\eta(f_j\chi_{B_R(0)})(u)-I_\lambda^\eta(f\chi_{B_R(0)})(u)|\ge k\}\cap B_M(0)\right|=o(1).
\end{equation}

On the other hand, we compute that for any fixed $m\in(1,Q/\lambda)$ by applying Minkowski's integral inequality,
\begin{eqnarray*}
&&\|I_\lambda(f_j\chi_{B_R(0)})-I_\lambda^\eta(f_j\chi_{B_R(0)})\|_m\\
&=&\left(\int_{\mathbb H^n}\left|\int_{B_\eta(u)}\frac{f_j(v)\chi_{B_R(0)}(v)}{|u^{-1}v|^\lambda}dv\right|^mdu\right)^{1/m}\\
&\le&\|f_j\chi_{B_R(0)}\|_1\left(\int_{B_\eta(v)}\frac{1}{|v^{-1}u|^{\lambda m}}du\right)^{1/m}\\
&\le&C(R,p,n)\eta^{(Q-\lambda m)/m}\\
&\rightarrow&0
\end{eqnarray*}
for every fixed $R$ as $\eta\rightarrow0$ since $Q>\lambda m$. That is,
\begin{equation}
\|I_\lambda(f_j\chi_{B_R(0)})-I_\lambda^\eta(f_j\chi_{B_R(0)})\|_m\le O(\eta).
\end{equation}

Similarly, we can derive the analogous statement for $f$,
\begin{equation}
\|I_\lambda(f\chi_{B_R(0)})-I_\lambda^\eta(f\chi_{B_R(0)})\|_m\le O(\eta).
\end{equation}
 
Also notice that 
\begin{equation}
\|I_\lambda(f_j)-I_\lambda(f_j\chi_{B_R(0)})\|_q\le C_{p,\lambda,n}\|f\chi_{B_R^c(0)}\|_p\le\epsilon(R)
\end{equation}
and
\begin{equation}
\|I_\lambda(f)-I_\lambda(f\chi_{B_R(0)})\|_q\le C_{p,\lambda,n}\|f\chi_{B_R^c(0)}\|_p\le\epsilon(R).
\end{equation}

Now returning to the estimate in (2.3), combining (2.4)--(2.8), we have for any $k>0$
\begin{eqnarray*}
&&|\{|I_\lambda(f_j)(u)-I_\lambda(f)(u)|\ge5k\}\cap B_M(0)|\\
&\le&|\{|I_\lambda(f_j)(u)-I_\lambda(f_j\chi_{B_R(0)})(u)|\ge k\}|+
|\{|I_\lambda(f_j\chi_{B_R(0)})(u)-I_\lambda^\eta(f_j\chi_{B_R(0)})(u)|\ge k\}|+\\
&&|\{|I_\lambda^\eta(f_j\chi_{B_R(0)})(u)-I_\lambda^\eta(f\chi_{B_R(0)})(u)|\ge k\}\cap B_M(0)|+\\
&&|\{I_\lambda^\eta(f\chi_{B_R(0)})(u)-I_\lambda(f\chi_{B_R(0)})(u)|\ge k\}|+|\{|I_\lambda(f_j\chi_{B_R(0)})(u)-I_\lambda(f)(u)|\ge k\}|\\
&\le&2\left[\frac{\epsilon(R)}{k}\right]^q+2\left[\frac{O(\eta)}{k}\right]^m+o(1).
\end{eqnarray*}

We can now conclude the convergence in measure of $\{f_j\}$ by properly choosing $\epsilon$, $R$, $M$, and $\eta$.
\end{proof}

Now we only need to verify the weak limit $f$ in preceding corollary satisfies $\|f\|_p=1$ to complete Theorem 1. We borrow Lemma 2.1 in \cite{Lio4} as follows to proceed.

\begin{lemma}
Let $f_j\rightarrow f$ weakly in $L^p(\mathbb H^n)$ and $I_\lambda(f_j)\rightarrow I_\lambda(f)$ weakly in $L^q(\mathbb H^n)$,
assume that (2.1) and (2.2) hold, and $|f_j|^p\rightarrow\mu$ and $|I_\lambda(f_j)|^q\rightarrow\nu$ weakly for two nonnegative measures $\mu$ and $\nu$ in $L^1(\mathbb H^n)$. Then, there exist two at most countable families (possibly empty) $\{u_j\}\subseteq\mathbb H^n$ and $\{k_j\}\subseteq(0,\infty)$ such that
\begin{equation}
\nu=|I_\lambda(f)|^q+\sum_jC_{p,\lambda,n}k_j^{q/p}\delta_{u_j}
\end{equation}
and
\begin{equation}
\mu\ge|f|^p+\sum_jk_j\delta_{u_j},
\end{equation}
in which $\delta_{u_j}$ is the Dirac function at $u_j$.
\end{lemma}

\begin{remark}
The original version of the above lemma is on $\R^N$, but it can be carried out as what we did in the proof of Lemma 4 on $\mathbb H^n$ because the crucial ingredient as Lemma 1.2 in \cite{Lio3} is valid in an arbitrary measure space. (See Remark 1.5 at the end of its proof.)
\end{remark}

With the help of Lemma 7, we now prove Theorem 1.

\begin{proof}[Proof of Theorem 1]
We show that $\|f\|_p=1$ by contradiction, then $f_j\rightarrow f$ strongly in $L^p(\mathbb H^n)$, which implies the theorem.

Observe that $\mu(\mathbb H^n)=1$ and $\nu(\mathbb H^n)=C_{p,\lambda,n}^q$ since $|f_j|^p\rightarrow\mu$ and $|I_\lambda(f_j)|^q\rightarrow\nu$ weakly in $L^1(\mathbb H^n)$. If $\|f\|_p^p=k<1$, then from (2.10),
$$\sum_jk_j\le\mu(\mathbb H^n)-\|f\|_p^p=1-k,$$
therefore by (2.9),
\begin{eqnarray*}
\nu(\mathbb H^n)&=&\|I_\lambda(f)\|_q^q+\sum_jC_{p,\lambda,n}^qk_j^{q/p}\\
&\le&C_{p,\lambda,n}^q\|f\|_p^q+C_{p,\lambda,n}^q\sum_jk_j^{q/p}\\
&\le&C_{p,\lambda,n}^qk^{q/p}+C_{p,\lambda,n}^q(\sum_jk_j)^{q/p}\\
&\le&C_{p,\lambda,n}^qk^{q/p}+C_{p,\lambda,n}^q(1-k)^{q/p}\\
&<&C_{p,\lambda,n}^q,
\end{eqnarray*}
which contradicts with the fact that $\nu(\mathbb H^n)=C_{p,\lambda,n}^q$, and we complete Theorem 1.
\end{proof}

\section{Proof of Theorem 2}

Theorem 2 is an analogue on $\mathbb H^n$ of \S 4.3 in \cite{LL}. The argument proceeds with no difficulty, and one completes the proof of Theorem 2 by noticing
$$|B_1(0)|=\frac{2\pi^\frac{Q-2}{2}\Gamma(1/2)\Gamma((Q+2)/4)}{(Q-2)\Gamma((Q-2)/2)\Gamma((Q+4)/4)}$$
from \cite{CL1}.

\section*{Acknowledgement}

The author would like to thank his advisor Professor Guozhen Lu for
bringing him into this interesting subject with continuous support
and encouragement on study and research. The author would also like to thank NSF Grant DMS 0901761 for the support.

\end{document}